\documentclass[12pt]{amsart}
\newtheorem{thm}{Theorem}[section]
\newtheorem{cor}[thm]{Corollary}
\newtheorem{lem}[thm]{Lemma}
\newtheorem{prop}[thm]{Proposition}
\theoremstyle{definition}
\newtheorem{defn}[thm]{Definition}
\theoremstyle{remark}
\newtheorem{rem}[thm]{Remark}
\numberwithin{equation}{section}

\def\cb{\mathcal{B}}

\def\bb{{\mathbb B}}
\def\bc{{\mathbb C}}

\def\bm{{\mathbb M}}
\def\bn{{\mathbb N}}

\def\br{{\mathbb R}}

\def\a{\alpha}
\def\b{\beta}
\def\g{\gamma}  
\def\d{\delta}  \def\D{\Delta}

\def\k{\kappa}

\def\p{\psi}
\def\n{\nu}

\def\s{\sigma} 
\def\t{\tau}
\def\f{\varphi}

\def\tr{\mathop{\rm Tr}}
\def\id{{\bf 1}\!\!{\rm I}}
\def\fb{{\mathbf{f}}}
\def\ab{{\mathbf{a}}}
\def\cb{{\mathbf{c}}}
\def\pb{{\mathbf{p}}}
\def\xb{{\mathbf{x}}}
\def\bx{{\mathbf{b}}}

\def\qb{{\mathbf{q}}}
\def\db{{\mathbf{d}}}

\def\D{\Delta}

\def\wb{{\mathbf{w}}}

\def\o{\otimes}

\def\a{\alpha}

\begin{document}
\title[Quantum quadratic operators]
{On quantum quadratic operators of $\bm_2(\mathbb{C})$ and their
dynamics}

\author{Farrukh Mukhamedov}
\address{Farrukh Mukhamedov\\
 Department of Computational \& Theoretical Sciences\\
Faculty of Science, International Islamic University Malaysia\\
P.O. Box, 141, 25710, Kuantan\\
Pahang, Malaysia} \email{{\tt far75m@yandex.ru}, {\tt
farrukh\_m@iiu.edu.my}}

\author{Hasan Ak\i n}
\address{Hasan Ak\i n, Department of Mathematics, Faculty of Education,
 Zirve University, Kizilhisar Campus, Gaziantep, 27260, Turkey}
\email{{\tt hasanakin69@gmail.com}}

\author{Seyit Temir}
\address{Seyit Temir, Department of Mathematics, Arts and
Science Faculty, Harran University, \c{S}anliurfa, 63120, Turkey}
\email{{\tt temirseyit@harran.edu.tr}}

\author{Abduaziz Abduganiev}
\address{Abduaziz Abduganiev\\
 Department of Computational \& Theoretical Sciences\\
Faculty of Science, International Islamic University Malaysia\\
P.O. Box, 141, 25710, Kuantan\\
Pahang, Malaysia} \email{{\tt azizi85@yandex.ru}}
\begin{abstract}

In the present paper we study nonlinear dynamics of quantum
quadratic operators (q.q.o) acting on the algebra of $2\times 2$
matrices $\bm_2(\bc)$. First, we describe q.q.o.  with Haar state as
well as quadratic operators with the Kadison-Schwarz property. By
means of such a description we provide an example of q.q.o. which is
not the Kadision-Schwarz operator. Then we study stability of
dynamics of q.q.o.

\end{abstract}

\maketitle

\section{Introduction}

It is known that there are many systems which are described by
nonlinear operators. One of the simplest nonlinear case is quadratic
one. Quadratic dynamical systems have been proved to be a rich
source of analysis for the investigation of dynamical properties and
modeling in different domains, such as population dynamics
\cite{Be,FG,HJ}, physics \cite{PL,UR}, economy \cite{D}, mathematics
\cite{HS,L1,V,U}. The problem of studying the behavior of
trajectories of quadratic stochastic operators was stated in
\cite{U}. The limit behavior and ergodic properties of trajectories
of such operators were studied in \cite{K,L1,L2,Ma,V}. However, such
kind of operators do not cover the case of quantum systems.
Therefore, in \cite{GM1,GM2} quantum quadratic operators acting on a
von Neumann algebra were defined and studied. Certain ergodic
properties of such operators were studied in \cite{M2,M3}. In those
papers basically dynamics of quadratic operators were defined
according to some recurrent rule (an analog of Kolmogorov-Chapman
equation) which makes a possibility to study asymptotic behaviors of
such operators. However, with a given quadratic operator one can
define also a non-linear operator whose dynamics (in non-commutative
setting) is not well studied yet. Note that in \cite{MM0} another
construction of nonlinear quantum maps were suggested and some
physical explanations of such nonlinear quantum dynamics were
discussed. There, it was also indicated certain applications to
quantum chaos. On the other hand, very recently, in \cite{FS}
convergence of ergodic averages associated with mentioned non-linear
operator are studied by means of absolute contractions of von
Neumann algebras. Actually, it is not investigated nonlinear
dynamics of convolution operators. Therefore, a complete analysis of
dynamics of quantum quadratic operator is not well studied.

In the present paper we are going to study nonlinear dynamics of
quantum quadratic operators (q.q.o.) acting on the algebra of
$2\times 2$ matrices $\bm_2(\bc)$. Since positive, trace-preserving
maps arise naturally in quantum information theory (see e.g.
\cite{N}) and other situations in which one wishes to restrict
attention to a quantum system that should properly be considered a
subsystem of a larger system with which it interacts. Therefore,
after preliminaries (Sec. 2) in section 3, we describe quadratic
operators with Haar state (invariant with respect to trace), namely
certain characterizations of q.q.o, Kadison-Schwarz operators
\footnote{In the literature the most tractable maps, the completely
positive ones, have proved to be of great importance in the
structure theory of C$^*$-algebras. However, general positive
(order-preserving) linear maps are very
intractable\cite{Kos,MM1,Ma,St}. It is therefore of interest to
study conditions stronger than positivity, but weaker than complete
positivity. Such a condition is called {\it Kadison-Schwarz
property}, i.e a map $\phi$ satisfies the Kadison-Schwarz property
if $\phi(a)^*\phi(a)\leq \phi(a^*a)$ holds for every $a$. Note that
every unital completely positive map satisfies this inequality, and
a famous result of Kadison states that any positive unital map
satisfies the inequality for self-adjoint elements $a$. In
\cite{Rob} relations between $n$-positivity of a map $\phi$ and the
Kadison-Schwarz property of certain map is established.}, which are
invariant w.r.t. trace, are given. By means of such a description in
Section 4, we shall provide an example of positive q.q.o. which is
not a Kadision-Schwarz operator. On the other hand, the such a
characterization is related to the separability condition, which
plays an important role in quantum information. It is worth to
mention that similar characterizations of positive maps defined on
$\bm_2(\bc)$ were considered in \cite{MM1,MM2}. Further, in section
4 we study stability of dynamics of quadratic operators. Note that
in \cite{M5} we have studied very simple dynamics of quadratic
operators. Moreover, we note that the considered quadratic operators
are related to quantum groups introduced in \cite{W}. Certain class
of quantum groups on $\bm_2(\bc)$ were investigated in \cite{S}.

\section{Preliminaries}

In what follows, by $\bm_2(\bc)$ we denote an algebra of $2\times 2$
matrices over complex field $\bc$. By $\bm_2(\bc)\o \bm_2(\bc)$ we
mean tensor product of $\bm_2(\bc)$ into itself. We note that such a
product can be considered as an algebra of $4\times 4$ matrices
$\bm_4(\bc)$ over $\bc$. In the sequel $\id$ means an identity
matrix, i.e. $ \id = \left( \begin{array}{cc} 1 & 0 \\ 0 & 1
\end{array} \right) $. By $S(\bm_2(\bc))$ we denote the set of all
states (i.e. linear positive functionals which take value 1 at
$\id$) defined on $\bm_2(\bc)$.

\begin{defn}\label{qqso} A linear operator $\D:\bm_2(\bc)\to \bm_2(\bc)\o\bm_2(\bc)$ is said to be
\begin{enumerate}
\item[(a)] -- a {\it quantum quadratic operator (q.q.o.)} if it
is unital (i.e. $\D\id=\id\o\id$) and positive ( i.e. $\D x\geq 0$
whenever $x\geq 0$);
\item[(b)] -- a {\it quantum convolution} if it is a q.q.o. and
satisfies coassociativity condition:
$$(\D\o
id)\circ\D=(id\o\D)\circ\D,$$ where $id$ is the identity operator of
$\bm_2(\bc)$;
\item[(c)] -- a {\it Kadison-Schwarz operator (KS)} if it satisfies
\begin{equation}\label{KS}
\D(x^*x)\geq\D(x)^*\D(x) \ \ \textrm{for all} \ x\in\bm_2(\bc) .
\end{equation}
\end{enumerate}
\end{defn}

One can see that if $\D$ is unital and KS operator, then it is a
q.q.o. A state  $h\in S(\bm_2(\bc))$ is called {\it a Haar state}
for a q.q.o. $\D$ if for every $x\in\bm_2(\bc)$ one has
\begin{equation}\label{Haar}
(h\o id)\circ \D(x)=(id\o h)\circ\D(x)=h(x)\id.
\end{equation}

\begin{rem}\label{qg} Note that if a quantum convolution $\D$ on
$\bm_2(\bc)$ becomes a $*$-homomorphic map with a condition
$$
\overline{\textrm{Lin}}((\id\o
\bm_2(\bc))\D(\bm_2(\bc)))=\overline{\textrm{Lin}}((\bm_2(\bc)\o\id)\D(\bm_2(\bc)))=\bm_2(\bc)\o\bm_2(\bc)
$$
then a pair $(\bm_2(\bc),\D)$ is called a {\it compact quantum
group} \cite{W,S}. It is known \cite{W} that for given any compact
quantum group there exists a unique Haar state w.r.t. $\D$.
\end{rem}

\begin{rem} Let $U:\bm_2(\bc)\o\bm_2(\bc)\to \bm_2(\bc)\o\bm_2(\bc)$
be a linear operator such that $U(x\o y)=y\o x$ for all $x,y\in
\bm_2(\bc)$. If a q.q.o. $\D$ satisfies $U\D=\D$, then $\D$ is
called a {\it quantum quadratic stochastic operator}. Such a kind of
operators were studied and investigated in \cite{M2,M5}.
\end{rem}

\begin{rem} We note that there is another approach to nonlinear quantum operators on $C^*$-algebras (see \cite{MM0}).
\end{rem}

Each q.q.o. $\D$ defines a conjugate operator $\D^*:(\bm_2(\bc)\o\bm_2(\bc))^*\rightarrow \bm_2(\bc)^*$  by
\begin{equation}\label{cqo}
\D^*(f)(x)=f(\D x), \ \ f\in (\bm_2(\bc)\o \bm_2(\bc))^*, \ x\in \bm_2(\bc).
\end{equation}

One can define an operator $V_\D$  by
\begin{equation}\label{qo} V_\D(\f)=\D^*(\f\o\f), \ \
\f\in S(\bm_2(\bc)),
\end{equation}
which is called a {\it quadratic operator (q.o.)}.  Note that
unitality and positivity of $\D$ imply that the operator $V_\D$ maps
$S(\bm_2(\bc))$ into itself. In some literature operator $V_\D$ is
called {\it quadratic convolution} (see for example \cite{FS}).

\section{Quantum quadratic operators on $\bm_2(\bc)$}

In this section we are going to describe quantum quadratic operators
on $\bm_2(\bc)$ as well as find necessary conditions for such
operators to satisfy the Kadison-Schwarz property.

Recall \cite{BR} that the identity and Pauli matrices $\{ \id,
\sigma_1, \sigma_2, \sigma_3 \}$ form a basis for $\bm_2(\bc)$,
where
\begin{eqnarray*}
\sigma_1 = \left( \begin{array}{cc} 0 & 1 \\ 1 & 0 \end{array}
\right)~~ \sigma_2 = \left( \begin{array}{cc} 0 & -i \\ i & 0
\end{array} \right)~~ \sigma_3 = \left( \begin{array}{cc} 1 & 0 \\
0 & -1 \end{array} \right).
\end{eqnarray*}

In this basis every matrix $x\in\bm_2(\bc)$ can  be written as $x =
w_0\id + \wb{\bf \sigma}$ with $w_0\in\bc$, $\wb =(w_1,w_2,w_3)\in
\bc^3$, here $\wb\s=w_1\s_1+w_2\s_2+w_3\s_3$. In what follows, we
frequently use notation
$\overline{\wb}=(\overline{w_1},\overline{w_2},\overline{w_3})$.

\begin{lem}\label{m2}\cite{RSW} The following assertions hold true:
\begin{enumerate}
\item[(a)] $x$ is self-adjoint iff  $w_0,\wb$  are reals; \item[(b)]
$\tr(x) = 1$ iff $w_0 =0.5$, here $\tr$ is the trace of a matrix
$x$;
\item[(c)] $x
> 0$ iff $\|\wb\|\leq w_0$, where
$\|\wb\|=\sqrt{|w_1|^2+|w_2|^2+|w_3|^2}$;
\item[(d)] A linear functional $\f$ on $\bm_2(\bc)$ is a state iff it can be represented
by
\begin{equation}\label{state}
{\f}(w_0\id + \wb\sigma)=w_0+\langle\wb,{\mathbf{f}}\rangle, \ \
\end{equation}
where ${\mathbf{f}}=(f_1,f_2,f_3)\in\br^3$ such that
$\|{\mathbf{f}}\|\leq 1$. Here as before $\langle\cdot,\cdot\rangle$
stands for the scalar product in $\bc^3$.
\end{enumerate}
\end{lem}

In the sequel we shall identify a state with a vector $\fb\in
\br^3$. By $\t$ we denote a normalized trace, i.e.
$$
\tau \left(%
\begin{array}{cc}
  x_{11}& x_{12} \\
  x_{21} & x_{22} \\
\end{array}%
\right)=\frac{x_{11}+x_{22}}{2},
$$
i.e. $\t(x)=\frac{1}{2}\tr(x)$, $x\in \bm_2(\bc)$,

Let  $\Delta:\bm_2(\bc)\rightarrow \bm_2(\bc) \otimes \bm_2(\bc)$ be
a q.q.o. Then we write the operator $\Delta $ in terms of a basis in
$\bm_2(\bc)\o\bm_2(\bc)$ formed by the Pauli matrices. Namely,
\begin{eqnarray}\label{D-bij}
&& \Delta \id=\id\otimes \id; \nonumber\\
&& \Delta (\sigma_i)=b_i(\id\otimes \id)+\overset{3}{\underset{j=1}{\sum
    }}b_{ji}^{(1)}(\id\otimes \sigma_j)+\overset{3}{\underset{j=1}{\sum
    }}b_{ji}^{(2)}(\sigma_j \otimes \id)+\overset{3}{\underset{m,l=1}{\sum
    }}b_{ml,i}(\sigma_m \otimes \sigma_l),
\end{eqnarray}
where $i=1,2,3$.

In general, a description of positive operators is one of the main
problems of quantum information. In the literature most tractable
maps are positive and trace-preserving ones, since such maps arise
naturally in quantum information theory (see \cite{N}).  Therefore, in the sequel we shall
restrict ourselves to q.q.o. which has a Haar state $\t$. So, we
would like to describe all such kind of maps.

\begin{prop}\label{trace-pre}
Let $\D:\bm_2(\bc)\to \bm_2(\bc)\o\bm_2(\bc) $ be a q.q.o. with a
Haar state $\t$, then in \eqref{D-bij} one has $b_j=0$, $b^{(1)}_{ij}=b^{(2)}_{ij}=0$ and
$b_{ij,k}$ are real for every $i,j,k\in\{1,2,3\}$. Moreover, $\D$
has the following form:
\begin{equation}\label{D3}
\D(x)=w_0\id\otimes\id+\sum_{m,l=1}^3\langle\bx_{ml},\overline{\wb}\rangle\sigma_m\otimes\sigma_l,
\end{equation}
where $x=w_0+\wb\s$, $\bx_{ml}=(b_{ml,1},b_{ml,2},b_{ml,3})$. Here
as before $\langle\cdot,\cdot\rangle $ stands for the standard
scalar product in $\bc^3$.
\end{prop}

\begin{proof}
From the positivity of $\Delta$ we get that $\Delta x^{*}=(\Delta
x)^{*}$, therefore
\begin{eqnarray*}
\Delta (\sigma_i^{*})&=&\overline{b_i}(\id\otimes
\id)+\overset{3}{\underset{j=1}{\sum
    }}\overline{b_{ji}^{(1)}}(\id\otimes \sigma_j)+\overset{3}{\underset{j=1}{\sum
    }}\overline{b_{ji}^{(2)}}(\sigma_j \otimes \id)+\overset{3}{\underset{m,l=1}{\sum
    }}\overline{b_{ml,i}}(\sigma_m \otimes \sigma_l).
\end{eqnarray*}

This yields that $b_i=\overline{b_i}$,
$b_{ji}^{(k)}=\overline{b_{ji}^{(k)}}$ ($k=1,2$) and
$b_{ml,i}=\overline{b_{ml,i}}$, i.e. all coefficients are real
numbers.

From \eqref{Haar} one finds
$$ \tau\otimes \tau(\Delta
x)=\tau(\tau\otimes id)(\Delta(x))=\tau(x)\t(\id)=\t(x), \ \ \ x\in
\bm_2(\bc),
$$
which means that $\t$ is an invariant state for $\D$. Hence, we
have $\t\o\t(\D(\s_i))=\t(\s_i)=0$ which yields $b_j=0$, $j=1,2,3$.

Again using the equality \eqref{Haar} with $h=\t$, one gets
\begin{eqnarray*} (id \otimes
\tau)\Delta(\sigma_i)&=&(id \otimes
\tau)\bigg[\overset{3}{\underset{j=1}{\sum}}\big(b_{ji}^{(1)}(\id\otimes
\sigma_j)+ b_{ji}^{(2)}(\sigma_j \otimes
\id)\big)+\overset{3}{\underset{m,l=1}{\sum}}b_{ml,i}(\sigma_m \otimes \sigma_l)\bigg]\\
&=& \overset{3}{\underset{j=1}{\sum}}b_{ji}^{(2)}\sigma_j=
\tau(\sigma_j)\id=0.
\end{eqnarray*}
Therefore, $b_{ji}^{(2)}=0$, for all $i,j=1,2,3.$ Similarly, one
finds
$$ (\tau \otimes
id)\Delta(\sigma_j)=\overset{3}{\underset{j=1}{\sum}}b_{ji}^{(1)}\sigma_j=
\tau(\sigma_j)\id,
$$
which means $b_{ji}^{(1)}=0$. Hence, $\D$ has the following form
\begin{equation}\label{d11}
\Delta(\sigma_i)=\overset{3}{\underset{m,l=1}{\sum
    }}b_{ml,i}(\sigma_m\otimes\sigma_l), \ \ i=1,2,3.
\end{equation}

Denoting
\begin{equation}\label{b-ml}
\bx_{ml}=(b_{ml,1},b_{ml,2},b_{ml,3}).
\end{equation}
and taking any $x=w_0\id+\wb\sigma\in \bm_2(\bc)$, from \eqref{d11}
we immediately find \eqref{D3}. This completes the proof.
\end{proof}

Let us turn to the positivity of $\D$. Given a vector
$\fb=(f_1,f_2,f_3)\in \br^3$ put
\begin{equation}\label{bij}
\b(\fb)_{ij}=\sum_{k=1}^3b_{ki,j}f_k.
\end{equation}
Define a matrix $\bb(\fb)=(\b(\fb)_{ij})_{ij=1}^3$, and by
$\|\bb(\fb)\|$ we denote its norm associated with Euclidean norm in
$\br^3$.

Given a state $\f$ by $E_\f$ we denote the canonical
conditional expectation defined by $E_{\varphi}(x\otimes
y)=\varphi(x)y$, where $x,y\in\bm_2(\bc)$.

In the sequel by $S$ we denote the unit ball in $\br^3$, i.e.
$$
S=\{\pb=(p_1,p_2,p_3)\in\br^3: \ p_1^2+p_2^2+p_3^2\leq 1\}.
$$
Let us denote
$$
\||\bb\||=\sup_{\fb\in S}\|\bb(\fb)\|. $$
\begin{prop}\label{positive}
Let $\Delta$ be a q.q.o. with a Haar state $\t$, then $\||\bb\||\leq
1$.
\end{prop}

\begin{proof} Let $x\in \bm_2(\bc)$ (i.e. $x=w_0\id+\wb \s$) be a positive element. Then for any state
$\f(x)=w_0+\langle\fb,\wb\rangle$ (here $\fb=(f_{1},f_{2},f_{3})\in
S$) from \eqref{D3},\eqref{bij} one finds
\begin{eqnarray*}
E_{\varphi}(\Delta(x))&=&
w_{0}\id+\sum\limits_{i,j=1}^{3}\langle\bx_{ij},\overline{\wb}\rangle
f_{i}\sigma_{j}\nonumber\\
&=&w_{0}\id+\bb(\fb)\wb\sigma\nonumber\\
\end{eqnarray*}
where we have used $\f(\sigma_{i})=f_{i}$ and
\begin{eqnarray*}
\sum\limits_{i=1}^{3}\langle\mathbf{b}_{ij},\overline{\wb}\rangle
f_{i}&=&\sum\limits_{l=1}^{3}\sum\limits_{i=1}^{3}b_{ij,l}f_{i}w_{l}\nonumber\\
&=&\sum\limits_{l=1}^{3}\beta_{jl}(\mathbf{f})w_{l}\nonumber\\
&=&(\bb(\fb)\wb)_{j}\nonumber
\end{eqnarray*}

Now positivity of $x$ yields that $E_\f(\D(x))$ is positive, for all
states $\f$, since $E_\f$ is a conditional expectation. Hence,
according to Lemma \ref{m2} positivity of $E_{\varphi}(\Delta(x))$
equivalent to $ \|\bb(\fb)\wb\|\leq w_{0}$ for all $\mathbf{f}$ and
$\wb$ with
 $\|\wb\|<w_{0}$.  Consequently, one finds that
$\|\bb(\fb)\|=\sup\limits_{\|\wb\|\leq 1}\|\bb(\fb)\wb\|\leq 1$,
which yields the assertion.
\end{proof}

\begin{rem} Note that similar characterizations of positive
maps defined on $\bm_2(\bc)$ were considered in \cite{MM2} (see also
\cite{Kos}). Characterization of completely positive mappings from
$\bm_2(\bc)$ into itself with invariant state $\t$ was established
in \cite{RSW}.
\end{rem}

Next we would like to find some conditions for q.q.o. to be
Kadison-Schwarz operators. To do it, we need the following auxiliary fact.

\begin{lem}\label{ac1}
Let $\ab,\cb\in \bc^3$. Then one has
\begin{eqnarray}\label{ac2}
&&(\ab\sigma)\cdot(\overline{\cb}\sigma)-(\cb\sigma)\cdot(\overline{\ab}\sigma)=
\big(\langle\ab,\cb\rangle-\langle\cb,\ab\rangle\big)\id+i\big([\ab,\overline{\cb}]+[\overline{\ab},\cb]\big)\sigma\\
\label{ac3}
&&(\ab\s)\cdot(\overline{\ab}\s)=\|\ab\|^2\id+i[\ab,\overline{\ab}]\s
\end{eqnarray}
 \end{lem}

The proof is straightforward.\\

Now introduce some notations. Given $x=w_0+\wb\s$ and a vector
$\fb\in S$ we denote
\begin{eqnarray}\label{x-ml}
&&x_{ml}=\langle\bx_{ml},\wb\rangle, \ \
\xb_{m}=\big(\langle\bx_{m1},\wb\rangle,\langle\bx_{m2},\wb\rangle,\langle\bx_{m3},\wb\rangle\big),\\[2mm]
\label{algam}
&&\a_{ml}=\langle\xb_{m},\xb_{l}\rangle-\langle\xb_{l},\xb_{m}\rangle,
\ \
\g_{ml}=[\xb_{m},\overline{\xb_{l}}]+[\overline{\xb_{m}},\xb_{l}],\\[2mm]
\label{qu}
&&\qb(\fb,\wb)=\big(\langle\b(\fb)_1,[\wb,\overline{\wb}]\rangle,\langle\b(\fb)_2,[\wb,\overline{\wb}]\rangle,
\langle\b(\fb)_3,[\wb,\overline{\wb}]\rangle\big),
\end{eqnarray}
where $\b(\fb)_m=\big(\b(\fb)_{m1},\b(\fb)_{m2},\b(\fb)_{m3}\big)$
(see \eqref{bij}) and as before
$\bx_{ml}=(b_{ml,1},b_{ml,2},b_{ml,3})$.

By $\pi$ we shall denote mapping $\{1,2,3,4\}$ to $\{1,2,3\}$
defined by $\pi(1)=2,\pi(2)=3,\pi(3)=1, \pi(4)=\pi(1)$.

\begin{thm}\label{ks3}
Let $\D:\bm_2(\bc)\to \bm_2(\bc)\o\bm_2(\bc) $ be a unital
Kadison-Schwarz operator with a Haar state $\t$, then it has the
form \eqref{D3} and the coefficients $\{b_{ml,k}\}$ satisfy the
following conditions
\begin{eqnarray}\label{ks11}
&&\|\wb\|^2-i\overset{3}{\underset{m=1}{\sum}} f_m\alpha_{\pi
(m),\pi (m+1)}-\overset{3}{\underset{m=1}{\sum}} \|\xb_{m}\|^2\geq
0\\
\label{ks2} &&
\bigg\|\qb(\fb,\wb)-i\overset{3}{\underset{m=1}{\sum}}
f_m\g_{\pi(m),\pi(m+1)}-[\xb_{m},\overline{\xb}_{m}]\bigg\|\leq
\|\wb\|^2-i\overset{3}{\underset{k=1}{\sum}} f_k\alpha_{\pi (k),\pi
(k+1)}-\overset{3}{\underset{m=1}{\sum}}\|\xb_{m}\|^2.
\end{eqnarray}
for all $\fb\in S,\wb\in \bc^3$.
\end{thm}
\begin{proof} Let $x\in{\bm}_2(\mathbb{C})$ be an arbitrary element, i.e. $x=w_0\id+\wb\cdot\s.$ Then one has
\begin{equation}\label{x*x}
x^*x=\big(|w_0|^2+\|\wb\|^2\big)\id+\big(w_0\overline{\wb}+\overline{w_0}\wb-i\big[\wb,\overline{\wb}\big]\big)\cdot\s.
\end{equation}
According to Proposition \ref{trace-pre} $\D$ has a form \eqref{D3},
therefore, taking into account \eqref{x*x} with \eqref{x-ml} one
finds
\begin{eqnarray}\label{DD1}
\Delta(x^{*}x)&=&(|w_{0}|^{2}+\|\wb\|^{2})\id+\sum\limits_{m,l=1}^{3}(\overline{w}_{0}\overline{{x}}_{m,l}+w_{0}x_{m,l})\sigma_{m}\otimes\sigma_{l}\nonumber\\
&&+
i\sum\limits_{m,l=1}^{3}\langle\mathbf{b}_{m,l},[\wb,\overline{\wb}]\rangle\sigma_{m}\otimes\sigma_{l}\\[2mm]
\label{DD2} \Delta(x)^{*}\Delta(x)
&=&|w_{0}|^{2}\id+\sum\limits_{m,l=1}^{3}(\overline{w}_{0}\overline{x}_{m,l}+w_{0}x_{m,l})\sigma_{m}\otimes\sigma_{l}\nonumber
\\
&&+\bigg(\sum\limits_{m,l=1}^{3}x_{m,l}\sigma_{m}\otimes\sigma_{l}\bigg)
\bigg(\sum\limits_{m,l=1}^{3}\overline{x_{m,l}}\sigma_{m}\otimes\sigma_{l}\bigg).
\end{eqnarray}

Noting that $\xb_m=(x_{m1},x_{m2},x_{m3})$, $m=1,2,3$ we rewrite the
last term of the equality \eqref{DD2} as follows

\begin{eqnarray*}
\bigg(\sum\limits_{m,l=1}^{3}x_{m,l}\sigma_{m}\otimes\sigma_{l}\bigg)
\bigg(\sum\limits_{m,l=1}^{3}\overline{x_{m,l}}\sigma_{m}\otimes\sigma_{l}\bigg)
&=&\bigg(\sum\limits_{m=1}^{3}\sigma_{m}\otimes(\xb_m\sigma)\bigg)
\bigg(\sum\limits_{m=1}^{3}\sigma_{m}\otimes(\overline{\xb}_m\s)\bigg)\nonumber\\
&=&\id\otimes\sum\limits_{k=1}^{3}(\xb_k\sigma)\cdot(\overline{\xb}_k\sigma)\nonumber\\
&&+i\sigma_{1}\otimes
\big((\xb_2\sigma)\cdot(\overline{\xb}_3\sigma)-(\xb_3\sigma)\cdot(\overline{\xb}_2\sigma)\big)\nonumber\\
&& +i\sigma_{2}\otimes
\big((\xb_3\sigma)\cdot(\overline{\xb}_1\sigma)-(\xb_1\sigma)\cdot(\overline{\xb}_3\sigma)\big)\nonumber\\
&&+ i\sigma_{3}\otimes
\big((\xb_1\sigma)\cdot(\overline{\xb}_2\sigma)-(\xb_2\sigma)\cdot(\overline{\xb}_1\sigma)\big)
\end{eqnarray*}

According to Lemma \ref{ac1} and \eqref{algam} the last equality
equals to
\begin{eqnarray}\label{XX}
X:&=&\id\otimes\bigg(\sum\limits_{j=1}^{3}\|\xb_{j}\|^{2}\id+i\sum\limits_{j=1}^{3}[\xb_{j},\overline{\xb}_{j}]\sigma)\bigg)\nonumber\\
&&+
i\sum_{m=1}^3\sigma_{m}\otimes\big(\a_{\pi(m),\pi(m+1)}\id+i\g_{\pi(m),\pi(m+1)}\s\big).
\end{eqnarray}

Then from \eqref{DD1}, \eqref{DD2} one gets
\begin{eqnarray}\label{DD-f}
\Delta(x^{*}x)-\Delta(x)^{*}\Delta(x)=
\|\wb\|^{2}\id+\sum\limits_{m,l=1}^{3}\langle\mathbf{b}_{ml}
,[\wb,\overline{\wb}]\rangle\sigma_{m}\otimes\sigma_{l}-X .
\end{eqnarray}

Now taking an arbitrary state $\f\in S(\bm_2(\bc))$ and applying
$E_\f$ to \eqref{DD-f} we have
\begin{eqnarray}\label{Ef}
E_\f(\Delta(x^{*}x)-\Delta(x)^{*}\Delta(x))=\|w\|^2\id+
i\overset{3}{\underset{m,l=1}{\sum}} \langle
\bx_{ml},[\wb,\overline{\wb}]\rangle f_m\sigma_l-E_\f(X),
\end{eqnarray}
where $\f(\s_m)=f_m$.

From \eqref{XX} one immediately finds
\begin{eqnarray}
E_\f(X)&=&\sum\limits_{m=1}^{3}\|\xb_{m}\|^{2}\id+i\sum\limits_{m=1}^{3}[\xb_{m},\overline{\xb}_{m}]\sigma\nonumber\\
\label{EX}
&&+
i\sum_{m=1}^3f_m(\a_{\pi(m),\pi(m+1)}\id+i\g_{\pi(m),\pi(m+1)}\s)
\end{eqnarray}

Now substituting the last equality \eqref{EX} to \eqref{Ef} with
\eqref{qu} we obtain

\begin{eqnarray*}
E_\f(\Delta(x^{*}x)-\Delta(x)^{*}\Delta(x))&=&
\bigg(\|\wb\|^2-i\overset{3}{\underset{m=1}{\sum}} f_m\alpha_{\pi
(m),\pi (m+1)}-\overset{3}{\underset{m=1}{\sum}}
\|\xb_{m}\|^2\bigg)\id\nonumber\\\label{Ef2}
&&+i\bigg(\qb(\fb,\wb)-i\overset{3}{\underset{m=1}{\sum}}
f_m\g_{\pi(m),\pi(m+1)}-[\xb_{m},\overline{\xb}_{m}]\bigg)\sigma.
\end{eqnarray*}

So, thanks to Lemma \ref{m2} the right hand side of \eqref{Ef2} is
positive if and only if \eqref{ks11} and \eqref{ks2} are satisfied
for all $\fb\in S,\wb\in \bc^3$. Note that the numbers $\a_{ml}$ are
skew-symmetric, i.e. $\overline{\a_{ml}}=-\a_{ml}$, therefore, the
equality \eqref{ks11} has a sense.
\end{proof}

Let us denote
$$
\textbf{h}(\wb)=\big(\langle\bx_{11},[\wb,\overline{\wb}]\rangle,\langle\bx_{12},[\wb,\overline{\wb}]\rangle,\langle\bx_{13},[\wb,\overline{\wb}]\rangle\big).
$$
Then one has the following

\begin{cor}\label{ksf}
Let $\D:\bm_2(\bc)\to \bm_2(\bc)\o\bm_2(\bc) $ be a Kadison-Schwarz
operator given by \eqref{D3}, then the coefficients $\{b_{ml,k}\}$
satisfy the following conditions
\begin{eqnarray}\label{ksf1}
&&\sum_{m=1}^3 \|\xb_{m}\|^2+i\a_{2,3}\leq\|\wb\|^2,\\[2mm]
\label{ksf2}
&&\bigg\|\textbf{h}(\wb)-i
\g_{2,3}+i\overset{3}{\underset{m=1}{\sum}}[\xb_{m},\overline{\xb}_{m}]\bigg\|\leq
\|\wb\|^2-i\a_{2,3}-\overset{3}{\underset{m=1}{\sum}}\|\xb_{m}\|^2,
\end{eqnarray} for all $\wb\in \bc^3$.
\end{cor}

The proof immediately follows from the previous Theorem \ref{ks3}
when we take $\fb=(1,0,0)$ in \eqref{ks11},\eqref{ks2}.

\begin{rem}
The provided characterization with \cite{MM3} allows us to construct
examples of positive or Kadison-Schwarz operators which are not
completely positive (see subsection 4.3).
\end{rem}

\section{Dynamics of quantum quadratic operators}

\subsection{General case}

In this section we are going to study dynamics of the quadratic
operator $V_\D$ associated with a q.q.o. $\D$ defined on
$\bm_2(\bc)$.

\begin{prop}\label{D*} Let $\D:\bm_2(\bc)\to \bm_2(\bc)\o\bm_2(\bc)$ be a linear operator
given by \eqref{D3}. Then the bilinear form $\D^*(\cdot\o\cdot)$ is
positive if and only if one holds
\begin{equation}\label{D*1}
\overset{3}{\underset{k=1}{\sum}}\bigg|\overset{3}{\underset{i,j=1}{\sum}}b_{ij,k}f_ip_j\bigg|^{2}\leq
1 \ \ \ \textrm{for all} \ \ \fb,\pb\in S.
\end{equation}
\end{prop}

\begin{proof} Take arbitrary states $\f,\p\in
S(\bm_2(\bc))$ and $\fb,\pb\in S$ be the corresponding vectors (see
\eqref{state}). Then from \eqref{D3} one finds
\begin{equation*}
\Delta^{*}(\varphi\otimes
\psi)(\s_k)=\overset{3}{\underset{i,j=1}{\sum}}b_{ij,k}f_ip_j, \ \ \
k=1,2,3.
\end{equation*}

Due to Lemma \ref{m2} (d) the functional $\Delta^{*}(\varphi\otimes
\psi)$ is a state if and only if the vector
$$
\textbf{f}_{\Delta^{*}(\varphi,
\psi)}=\bigg(\overset{3}{\underset{i,j=1}{\sum}}b_{ij,1}f_ip_j,
\overset{3}{\underset{i,j=1}{\sum}}b_{ij,2}f_ip_j,\overset{3}{\underset{i,j=1}{\sum}}b_{ij,3}f_ip_j\bigg).
$$
satisfies $\|\fb_{\D^*(\f,\p)}\|\leq 1$, which is the required
assertion.

\end{proof}

From the proof of Propositions \ref{positive} and \ref{D*} we get
the following

\begin{cor}\label{qc2} Let $\bb(\fb)$ be the corresponding matrix to an
operator given by \eqref{D3}. Then $\||\bb\||\leq 1$ if and only if
\eqref{D*1} is satisfied.

\end{cor}

Let us find some sufficient condition for the coefficients
$\{b_{ij,k}\}$ to satisfy \eqref{D*1}.

\begin{cor}\label{D*2} Let
\begin{equation}\label{D*3}
\overset{3}{\underset{i,j,k=1}{\sum}}|b_{ij,k}|^{2}\leq 1
\end{equation}
be satisfied, then \eqref{D*1} holds.
\end{cor}

\begin{proof} Let \eqref{D*3} be satisfied. Take any $\fb,\pb\in
S$, then
\begin{eqnarray*}
\bigg|\overset{3}{\underset{i,j=1}{\sum}}b_{ij,k}f_ip_j\bigg|^{2}&\leq&
\bigg(\overset{3}{\underset{i,j=1}{\sum}}|b_{ij,k}||f_ip_j|\bigg)^{2}\\
&\leq&\overset{3}{\underset{i,j=1}{\sum}}|b_{ij,k}|^{2}
\overset{3}{\underset{i=1}{\sum}}|f_i|^{2}\overset{3}{\underset{j=1}{\sum}}|p_j|^{2}\\
&\leq&\overset{3}{\underset{i,j=1}{\sum}}|b_{ij,k}|^{2}
\end{eqnarray*}
which implies the assertion.
\end{proof}

Let us consider the quadratic operator, which is defined by
$V_\D(\f)=\D^*(\f\o\f)$, $\f\in S(\bm_{2}(\bc))$. According to
Proposition \ref{D*} and Corollary \ref{qc2} we conclude that the
operator $V_\D$ maps $S(\bm_{2}(\bc))$ into itself if and only if
$\||\bb\||\leq 1$. To study the dynamics of $V_\D$ on $S(\bm_{2}(\bc))$ it is enough
to investigate behaviour of the corresponding vector $\fb_{V_\D(\f)}$ in $\br^3$. Therefore, from \eqref{D3} we find that
\begin{eqnarray*}
V_\Delta(\varphi)(\s_k)=\overset{3}{\underset{i,j=1}{\sum}} b_{ij,k}f_if_j, \ \ \fb\in S.
\end{eqnarray*}

This suggests us the consideration of a nonlinear operator $V:S\to
S$ defined by
\begin{equation}\label{V}
V(\textbf{f})_{k}=\overset{3}{\underset{i,j=1}{\sum}}b_{ij,k}f_{i}f_{j},
  \ \ \  k=1,2,3.
\end{equation}
where $\fb=(f_1,f_2,f_3)\in S$. Furthermore, we are going to study
dynamics of $V$.

Since $S$ is a convex compact set, then due to Schauder theorem $V$
has at least one fixed point. One can see that one of the fixed
points is $(0,0,0)$. Furthermore, we will be interested on
uniqueness (stability ) of this fixed point.

Denote
\begin{equation}\label{alp}
\a_k=\sqrt{\sum_{j=1}^3\bigg(\sum_{i=1}^3|b_{ij,k}|\bigg)^{2}}+
\sqrt{\sum_{i=1}^3\bigg(\sum_{j=1}^3|b_{ij,k}|\bigg)^{2}}, \ \ \
\a=\sum_{k=1}^3\a_k^2
\end{equation}

\begin{thm}\label{alfa} If $\a<1$ then $V$ is a contraction, hence $(0,0,0)$ is a unique stable fixed point.
\end{thm}
\begin{proof}
Let us take $\fb,\pb\in S$ and consider the difference
\begin{eqnarray*}
|V(\textbf{f})_{k}-V(\textbf{p})_{k}|&\leq& \overset{3}{\underset{i,j=1}{\sum}}|b_{ij,k}||f_{i}f_{j}-p_{i}p_{j}|\\
&\leq&
\overset{3}{\underset{i,j=1}{\sum}}|b_{ij,k}||f_{i}||f_{j}-p_{j}|+
\overset{3}{\underset{i,j=1}{\sum}}|b_{ij,k}||p_{j}||f_{i}-p_{i}|\\
&\leq&\overset{3}{\underset{i,j=1}{\sum}}|b_{ij,k}||f_{j}-p_{j}|+
\overset{3}{\underset{i,j=1}{\sum}}|b_{ij,k}||f_{i}-p_{i}|\\
&\leq&
\bigg(\sqrt{\sum_{j=1}^3\bigg(\sum_{i=1}^3|b_{ij,k}|\bigg)^{2}}+
\sqrt{\sum_{i=1}^3\bigg(\sum_{j=1}^3|b_{ij,k}|\bigg)^{2}}\bigg)\|\fb-\pb\|\\
&=&\a_k \|\fb-\pb\|,
\end{eqnarray*}
where $k=1,2,3.$ Hence, $V$ is a contraction, so it has a unique
fixed point. This completes the proof.
\end{proof}

Note that the condition $\a<1$ in Theorem \ref{alfa} is too strong,
therefore, it would be interesting to find more weaker conditions
than the provided one.

Put
\begin{equation}\label{delta123}
\d_k=\sum_{i,j=1}^3 |b_{ij,k}|, \ \ \ k=1,2,3.
\end{equation}
and denote $\db=(\d_1,\d_2,\d_3)$.

Given a quadratic operator $V$ by \eqref{V} define a new operator $\tilde V:\br^3\to\br^3$ by
\begin{equation}\label{nV}
\tilde V(\pb)_k=\overset{3}{\underset{i,j=1}{\sum}}|b_{ij,k}|p_ip_j,
\ \ \pb\in\br^3, \ k=1,2,3.
\end{equation}

For any given $\fb\in S$, we denote $\g_\fb=\max\{|f_1|,|f_2|,|f_3|\}$.
It is clear that $\g_\fb\leq 1$.

\begin{prop}\label{g-f-V} If the sequence $\{\tilde V^n(\db)\}$ is
bounded, then for any $\fb\in S$ with $\g_\fb<1$ one has $V^n(\fb) \to (0,0,0)$ as $n\to \infty$.
\end{prop}

\begin{proof} From \eqref{V} we immediately find
\begin{equation*}
|V(\fb)_k|\leq \g_\fb^2\sum_{i,j=1}^3|b_{ij,k}|=\g^2_\fb \d_k, \ \ \ k=1,2,3.
\end{equation*}
Hence, the last inequality implies that
\begin{equation}\label{V-1}
|V^2(\fb)_k|\leq \sum_{i,j=1}^3|b_{ij,k}||V(\fb)_i||V(\fb)_j|\leq
\g^{2^2}_\fb\tilde V(\db)_k, \ \ \ k=1,2,3,
\end{equation}
here as before $\db=(\d_1,\d_2,\d_3)$.

Hence, using mathematical induction one can get
\begin{equation}\label{Vn}
|V^n(\fb)_k|\leq  \g^{2^n}_\fb\tilde V^{n-1}(\db)_k, \ \ \textrm{for any} \  n\in\bn,\  k=1,2,3
\end{equation}

Due to $\g_\fb<1$ and boundedness of $\{\tilde V^n(\db)_k\}$, from  \eqref{Vn} we obtain the desired assertion.
\end{proof}

Next Lemma provides us some sufficient condition for the boundedness of $\{\tilde V^n(\db)_k\}$.

\begin{lem}\label{bb1} Assume that one has
\begin{equation}\label{bb2}
\sum_{i,j=1}^3|b_{ij,k}|\leq 1 , \ \ k=1,2,3.
\end{equation}
Then the sequence
$\{\tilde V^n(\db)_k\}$ is bounded.
\end{lem}

\begin{proof} From \eqref{bb2} we conclude that $\d_k\leq 1$ for
every $k=1,2,3$. Therefore, it follows from  \eqref{nV} that
$$
|\tilde V(\db)_k|=\sum_{i,j=1}^3|b_{ij,k}|\d_i\d_j\leq \d_k\leq 1.
$$
Now assume that $|\tilde V^{m}(\db)_k|\leq \d_k$ for every $k=1,2,3$. Then, due to assumption, from \eqref{bb2}  one gets
\begin{eqnarray*}
|\tilde V^{m+1}(\db)_k|&=&\sum_{i,j=1}^3|b_{ij,k}||\tilde V^{m}(\db)_i||\tilde V^{m}(\db)_j|\\
&\leq &\sum_{i,j=1}^3|b_{ij,k}|\d_i\d_j\\
&\leq& \d_k.
\end{eqnarray*}
Hence, the mathematical induction implies that $|\tilde V^n(\db)_k|\leq \d_k$ for every $n\in\bn$,
$k=1,2,3$. This completes the proof.
\end{proof}

Now we are interested when the sequence $\{\tilde V^n(\db)\}$
converges to $(0,0,0)$.

\begin{lem}\label{bb33} Assume that \eqref{bb2} is satisfied. If there is $n_0\in\bn$ such that $\tilde
V^{n_0}(\db)_k<1$ for every $k=1,2,3$, then $\tilde V^n(\db)\to
(0,0,0)$ as $n\to \infty$;
\end{lem}

\begin{proof} Let us denote $v=\max\{V^{n_0}(\db)_1,V^{n_0}(\db)_k,V^{n_0}(\db)_3\}$, then due to the assumption
one has $0<v<1$. Then from \eqref{nV} with \eqref{bb2} one gets
$$
\tilde V^{n_0+1}(\db)_k=\sum_{i,j=1}^3|b_{ij,k}|V^{n_0}(\db)_i
V^{n_0}(\db)_j\leq v^2\d_k\leq v^2 .
$$
Iterating this procedure we obtain $\tilde V^{n+n_0}(\db)_k\leq
v^{2^n}$ for every $n\in\bn$, $k=1,2,3$. This yields the assertion.
\end{proof}

Now we are ready to formulate a main result about stability of the
unique fixed point $(0,0,0)$ for $V$.

\begin{thm}\label{bb-main} Assume that \eqref{bb2} is satisfied.
 If there is  $k_0\in\{1,2,3\}$ such that $\d_{k_0}<1$
and for each $k=1,2,3$  one can find $i_0\in\{1,2,3\}$ with
$|b_{i_0,k_0,k}|+|b_{k_0,i_0,k}|\neq 0$, then $(0,0,0)$ is a unique
stable fixed point, i.e. for every $\fb\in S$ one has $V^n(\fb)\to
(0,0,0)$ as $n\to \infty$.
\end{thm}

\begin{proof} Take any $k\in \{1,2,3\}$, then due to condition one can find $i_0$ such
that $|b_{i_0,k_0,k}|+|b_{k_0,i_0,k}|\neq 0$. Then from \eqref{nV}
with \eqref{bb2} we have
\begin{eqnarray*}
\tilde V(\db)_k&=& \sum_{i,j=1}^3|b_{ij,k}|\d_j\d_j\\
&=&
\sum_{j=1}^3|b_{k_0j,k}|\d_{k_0}\d_j+\sum_{i=1}^3|b_{ik_0,k}|\d_{i}\d_{k_0}+
\sum_{i,j=1\atop i,j\neq k_0}^3|b_{ij,k}|\d_{i}\d_j-|b_{k_0k_0,k}|\d_{k_0}^2\\
&\leq&
\sum_{j=1}^3|b_{k_0j,k}|\d_{k_0}+\sum_{i=1}^3|b_{ik_0,k}|\d_{k_0}+
\sum_{i,j=1\atop i,j\neq k_0}^3|b_{ij,k}|-|b_{k_0k_0,k}|\d_{k_0}^2\\
&=&\d_k-(1-\d_{k_0})\sum_{j=1}^3(|b_{k_0j,k}|+|b_{jk_0,k}|)+|b_{k_0k_0,k}|(1-\d_{k_0}^2)\\
&=&\d_k-(1-\d_{k_0})\bigg(\sum_{j=1}^3(|b_{k_0j,k}|+|b_{jk_0,k}|)-(1+\d_{k_0})|b_{k_0k_0,k}|\bigg)\\
&=&\d_k-(1-\d_{k_0})\bigg(\sum_{j=1\atop j\neq k_0}^3(|b_{k_0j,k}|+|b_{jk_0,k}|)+(1-\d_{k_0})|b_{k_0k_0,k}|\bigg)\\
&=&\d_k-(1-\d_{k_0})\sum_{j=1\atop j\neq k_0}^3(|b_{k_0j,k}|+|b_{jk_0,k}|)-(1-\d_{k_0})^2|b_{k_0k_0,k}|\\
&<&\d_k\leq 1,
\end{eqnarray*}
hence from Lemma \ref{bb33} we find that $\tilde V^n(\db)\to
(0,0,0)$ as $n\to \infty$. So, from \eqref{Vn} one gets the desired
assertion.
\end{proof}

\subsection{Diagonal case} In this subsection we are going to
investigate more concrete case called diagonal operators.

We call a quadratic operator $V$ given by \eqref{V} is
{\it diagonal} if $b_{ij,k}=0$ for all $i,j$ with $i\neq j$. In what
follows, for the sake of shortness we write $b_{ik}$ instead of
$b_{ii,k}$. Hence from \eqref{V} we derive
\begin{equation}\label{V2}
(V(\fb))_k=\sum_{i=1}^3b_{ik}f_i^2, \ \ \fb=(f_1,f_2,f_3)\in S.
\end{equation}

First we are interested when $V$ maps $S$ into itself, i.e. $V(S)\subset S$. If the coefficients $\{b_{ik}\}$ satisfy
\eqref{D*1} then from Proposition \ref{D*} we conclude the desired inclusion. Next lemma provides us a sufficient condition to $\{b_{ik}\}$ for the satisfaction of \eqref{D*1}.

\begin{lem}\label{d*11} Let $V$ be a diagonal quadratic operator given by \eqref{V2}. Assume that one holds
\begin{equation}\label{bb3}
\sum_{k=1}^3\max_i\{|b_{ik}|^2\}\leq 1,
\end{equation}
then \eqref{D*1} is satisfied.
\end{lem}

\begin{proof} Let us check \eqref{D*1}. Take any $\fb,\pb\in S$,
then taking into account the definition of diagonal operator and our
notation we get
\begin{eqnarray*}\label{b1}
\bigg|\sum_{i,j=1}^3b_{ij,k}f_ip_j\bigg|&\leq&
\sum_{i=1}^3|b_{ik}||f_i||p_i|\nonumber\\
&\leq&\max_i\{|b_{ik}|\}\sum_{i=1}^3|f_i||p_i|\nonumber\\
&\leq&\max_i\{|b_{ik}|\}\|\fb\|\|\pb\|\nonumber\\
&\leq& \max_i\{|b_{ik}|\},
\end{eqnarray*}
which implies the desired inequality.
\end{proof}

\begin{rem} It is easy to see that the condition \eqref{bb3} is weaker than \eqref{D*3}.
\end{rem}

\begin{thm} Let $V$ be a diagonal quadratic operator given by \eqref{V2}.   Assume that
\begin{eqnarray}\label{bb4}
\sum_{k=1}^3\max_i\{|b_{i,k}|^2\}<1,
\end{eqnarray}
then the operator has a unique stable fixed point $(0,0,0)$.\\
\end{thm}

\begin{proof} First, from \eqref{bb4} with  Lemma \ref{d*11} we conclude that $V$ maps $S$ into itself.
 Now denote
$a_k:=\max\limits_i\{|b_{i,k}|\}$ and put
\begin{eqnarray}\label{b3}
\g:=\sum_{k=1}^3 a_k^2.
\end{eqnarray}

Take any $\fb=(f_1,f_2,f_3)\in S$. Then from \eqref{V2} we find
$$
|V(\fb)_k|\leq \sum_{i=1}^3|b_{ik}|f_i^2\leq a_k\sum_{i=1}^3f_i^2\leq a_k, \ \ k=1,2,3.
$$
From the last inequality with \eqref{V2} implies
\begin{eqnarray*}\label{b4}
|V^2(\fb)_k|\leq a_k\g, \ \ \ k=1,2,3.
\end{eqnarray*}
Now iterating this procedure, we derive
\begin{eqnarray}\label{b5}
|V^n(\fb)_k|\leq a_k\g^{n-1}, \ \ \ k=1,2,3.
\end{eqnarray}
for every $n\geq 2$. Due to \eqref{bb4} we have  $\g<1$, therefore \eqref{b5} implies that $V^n(\fb)\to 0$ as
$n\to\infty$. Arbitrariness of $\fb$ proves the theorem.
\end{proof}

\begin{rem} Note that if \eqref{bb4} is not satisfied, then the corresponding quadratic operator
may have more than one fixed points. Indeed, let us consider the following diagonal operator
defined by $V_0(\fb)=(f_1^2,0,0)$, where $\fb=(f_1,f_2,f_3)$. One can see that for this operator \eqref{bb3} is satisfied, but
\eqref{bb4} does not hold. It is clear that $V_0$ has two fixed points such as $(1,0,0)$ and $(0,0,0)$. \end{rem}

\subsection{Example of diagonal quadratic operator which is
not KS one}  In this subsection we are going to provide an example of a diagonal operator for which \eqref{bb3} is not satisfied, but nevertheless it maps $S$ into itself. Moreover, we shall show to such an operator does not satisfy the KS property in certain values of the coefficients.

Let us consider the following diagonal quadratic
operator defined by
\begin{equation}\label{q1}
\left\{
\begin{array}{lll}
(V(\fb))_1=f_1^2,\\
(V(\fb))_2=af_2^2+bf_3^2,\\
(V(\fb))_3=cf_3^2,
\end{array}
\right. \ \ \fb=(f_1,f_2,f_3).
\end{equation}

We can immediately observe that for given operator \eqref{bb3} is
not satisfied since $b_{11}=1$ and if one of the coefficients $a,b,c$ is non zero.

\begin{lem}\label{abc} Let
\begin{equation}\label{bb5} \max\{a^2,b^2\}+c^2\leq1
\end{equation}
be satisfied. Then for the quadratic operator \eqref{q1} the condition
\eqref{D*1} is satisfied..
\end{lem}

\begin{proof} Take any $\fb,\pb\in S$, and denote
$$
z=|f_2p_2|+|f_3p_3|.
$$
Then using $|f_1p_1|+|f_2p_2|+|f_3p_3|\leq1
$ we have
\begin{eqnarray}\label{abc2}
\sum\limits_{k=1}^3\bigg|
\sum\limits_{m,l=1}^3b_{ml,k}f_mp_l\bigg|^2-1&=&|f_1p_1|^2+|af_2p_2+bf_3p_3|^2+|cf_3p_3|^2-1\nonumber\\[2mm]
&\leq&|f_1p_1|^2+\max\{a^2,b^2\}(|f_2p_2|+|f_3p_3|)^2+c^2|f_3p_3|-1\nonumber\\[2mm]
&\leq&(f_1p_1)^2+\max\{a^2,b^2\}(|f_2p_2|+|f_3p_3|)^2+c^2(|f_2p_2|+|f_3p_3|)-1\nonumber\\[2mm]
&\leq&(1-|f_2p_2|-|f_3p_3|)^2+\max\{a^2,b^2\}(|f_2p_2|+|f_3p_3|)^2\nonumber\\
&&+c^2(|f_2p_2|+|f_3p_3|)-1\nonumber\\[2mm]
&\leq&(1-z)^2+\max\{a^2,b^2\}z^2+c^2z-1\nonumber\\
&=&z\big(z(1+\max\{a^2,b^2\})+c^2-2\big).
\end{eqnarray}

Due to $0\leq{z}\leq1$, we conclude that \eqref{abc2} is less than zero, if one has
$$
\max\{a^2,b^2\}+c^2-1\leq 0,
$$
which implies the assertion.
\end{proof}

The proved lemma implies that the operator \eqref{q1} maps $S$ into
itself. Therefore, let us examine dynamics of \eqref{q1} on $S$.

\begin{thm} Let $V$ be a quadratic operator given by \eqref{q1}, and assume \eqref{bb5} is satisfied. Then the
following assertions hold true:
\begin{enumerate}
\item[(i)] $(0,0,0)$, $(1,0,0)$ are fixed points of $V$;
\item[(ii)] if $|f_1|=1$, then $V^n(\fb)=(1,0,0)$ for all $n\in\bn$;
\item[(iii)] Let $|c|=1$, then there is another fixed point $(0,0,c)$. Moreover, if $|f_3|=1$, then $V^n(\fb)=(0,0,c)$ for every $n\geq 2$, and if $\max\{|f_1|,|f_3|\}<1$, then $V^n(\fb)\to (0,0,0)$ as $n\to \infty$;
\item[(iv)] Let $|a|=1$, then there is another fixed point $(0,a,0)$. If $|af_2^2+bf^2_3|=1$, then
$V^n(\fb)=(0,a,0)$ for all $n\geq 2$, and if $|af_2^2+bf^2_3|<1$ and $|f_1|<1$, then
$V^n(\fb)=(0,0,0)$ as $n\to\infty$;
\item[(v)] Let $|b|=1$, $|a|<1$. If $|f_1|<1$, then  $V^n(\fb)=(0,0,0)$ as $n\to\infty$;
\item[(vi)] Let $\max\{a^2,b^2\}+c^2<1$. If  $|f_1|<1$, then
$V^n(\fb)=(0,0,0)$ as $n\to\infty$.
\end{enumerate}
\end{thm}
\begin{proof} The statements (i) and (ii) are obvious. Hence, furthermore, we assume $|f_1|<1$. Now let us consider (iii). If $|c|=1$, then from \eqref{bb5} one gets that
$a=b=0$. Hence, in this case, we have another fixed point $(0,0,c)$.
One can see that $V(0,0,-c)=(0,0,c)$. So, if $|f_3|=1$, then
$V^n(\fb)=(0,0,c)$ for every $n\geq 2$. If $\max\{|f_1|,|f_3|\}<1$,
then from \eqref{q1} we find $V^n(\fb)=(f_1^{2^n},0,c
f_3^{2^n})\to(0,0,0)$ as $n\to\infty$.

(iv). Let $|a|=1$, then from \eqref{bb5} one
finds $c=0$, which implies the existence of another fixed point $(0,a,0)$. From \eqref{q1} we find
\begin{equation}\label{v-222}
(V^n(\fb))_2=a(af_2^2+bf^2_3)^{2^{n-1}}.
\end{equation}
Hence, if $|af_2^2+bf^2_3|=1$ then
$V^n(\fb)=(0,a,0)$ for every $n\geq 2$. If $|af_2^2+bf^2_3|<1$, $|f_1|<1$ then
$V^n(\fb)\to(0,0,0)$ as $n\to\infty$.

(v).  Let $|b|=1$, $|a|<1$, then we have $c=0$. In this case, one has $|af_2^2+bf^2_3|<1$ for every $\fb\in S$, therefore, \eqref{v-222} yields the desired assertion.

(vi). Let us assume that $\max\{a^2,b^2\}+c^2<1$. Then
modulus of all the coefficients are strictly less than one. For the
sake of simplicity denote $m=\max\{|a|,|b|\}$. From \eqref{q1} we
have
\begin{equation}\label{q112}
\left\{
\begin{array}{ll}
|(V(\fb))_2|\leq m, \\
|(V(\fb))_3|\leq |c|,
\end{array}
\right.
\end{equation}
for every $\fb\in S$.

Then denoting $\k=m^2+|c|^2$,  from \eqref{q1} with \eqref{q112} one gets
\begin{equation}\label{q12}
\left\{
\begin{array}{ll}
|(V^2(\fb))_2|\leq m\k\\
|(V^2(\fb))_3|\leq |c|^3
\end{array}
\right.
\end{equation}
Assume that
\begin{equation}\label{q1m}
\left\{
\begin{array}{ll}
|(V^m(\fb))_2|\leq m\k^{2^{m-1}-1}\\
|(V^m(\fb))_3|\leq |c|^{2^{m+1}-1}
\end{array}
\right.
\end{equation}
for some $m\geq 2$. Then from \eqref{q1} with \eqref{q1m} we derive
\begin{eqnarray*}
|(V^{m+1}(\fb))_2|&\leq& m\big(m^2\k^{2^m-2}+|c|^{2^{m+2}-2}\big)\\
&=& m\big(m^2\k^{2^m-2}+(|c|^2)^{2^{m}-2}|c|^{2^{m+1}+2}\big),\\
&\leq& m\big(m^2\k^{2^m-2}+\k^{2^{m}-2}|c|^2\big),\\
&=& m \k^{2^m-1},
\end{eqnarray*}
here we have used $|c^2|\leq \k$.

One can see that
$$|(V^{m+1}(\fb))_3|\leq |c|^{2^{m+2}-1}.
$$
Consequently, by the induction we conclude that \eqref{q1m} is valid for all $m\geq 2$.

According to our assumption  one has $\k<1$, therefore, \eqref{q1m} with \eqref{q1} implies that
$V^n(\fb)\to(0,0,0)$ ($n\to\infty$) when $|f_1|<1$.
\end{proof}

By $\D_{a,b,c}$ we denote a linear operator from $M_2(\bc)$ to $M_2(\bc)\otimes M_2(\bc)$ corresponding to \eqref{q1}.
Now we would like to choose parameters $a,b,c$ so that $\D_{a,b,c}$ is not KS-operator.

\begin{thm} Assume that \eqref{bb5} is satisfied. If $|a|+|b|>1$, then
$\D_{a,b,c}$ is not KS-operator.
\end{thm}
\begin{proof} It
is enough choose the numbers $a,b,c$ so that, for them the conditions of Corollary
\ref{ksf} are not satisfied. Let us start to look to \eqref{ksf1}. A
little calculations show that
\begin{eqnarray}\label{x123}
&&\xb_1=(\overline{w}_1,0,0),\xb_2=(0,a\overline{w}_2,0),
\xb_3=(0,0,b\overline{w}_2+c\overline{w}_3),
\end{eqnarray}
where $(w_1,w_2,w_3)\in\bc^3$. So, from \eqref{algam} we immediately find
\begin{eqnarray*}
&&\a_{2,3}=\langle\xb_2,\xb_3\rangle-\langle\xb_3,\xb_2\rangle=0.
\end{eqnarray*}
Hence, from the last equality with \eqref{x123} we infer that
\eqref{ksf1} is reduced to
\begin{eqnarray}\label{e12}
|a|^2|w_2|^2+|b\overline{w}_2+c\overline{w}_3|^2\leq|w_2|^2+|w_3|^2
\end{eqnarray}

Now let us estimate left hand side the expression of \eqref{e12}.
\begin{eqnarray*}
|a|^2|w_2|^2+|b\overline{w}_2+c\overline{w}_3|^2&\leq&|a|^2|w_2|^2+\bigg(|b||w_2|+|c||w_3|\bigg)^2\\
&\leq&|a|^2|w_2|^2+\max\{|b|^2,|c|^2\}\bigg(|w_2|+|w_3|\bigg)^2\\
&\leq&|a|^2|w_2|^2+2\max\{|b|^2,|c|^2\}\bigg(|w_2|^2+|w_3|^2\bigg)
\end{eqnarray*}

Hence, if one holds
\begin{eqnarray}\label{e13}
|a|^2|w_2|^2+2\max\{|b|^2,|c|^2\}\bigg(|w_2|^2+|w_3|^2\bigg)\leq
|w_2|^2+|w_3|^2
\end{eqnarray}
then surely \eqref{e12} is satisfied. Therefore, let us examine
\eqref{e13}. From \eqref{e13} one finds
\begin{eqnarray*}
&&\bigg(1-|a|^2-2\max\{|b|^2,|c|^2\}\bigg)|w_2|^2+\bigg(1-2\max\{|b|^2,|c|^2\}\bigg)|w_3|^2\geq0,
\end{eqnarray*}
which is satisfied if one has
\begin{eqnarray}\label{e14}
|a|^2+2\max\{|b|^2,|c|^2\}\leq1.
\end{eqnarray}

Now let us look to the condition \eqref{ksf2}. From \eqref{x123}
direct calculations shows us that
\begin{eqnarray}
\label{hw} \left\{
\begin{array}{lll}
\textbf{h}(\wb)=\big(\overline{w}_2w_3-\overline{w}_3w_2,0,0\big)\\[2mm]
\g_{2,3}=\big(2ab|w_2|^2+ac(\overline{w}_2w_3+w_2\overline{w}_3),0,0\big)\\[2mm]
\sum\limits_{m=1}^3[\xb_m,\overline{\xb}_m]=0\\
\end{array}
\right.
\end{eqnarray}

Therefore, the left hand side of \eqref{ksf2} can be written as follows
\begin{eqnarray*}
\bigg\|\textbf{h}(\wb)-i
\g_{2,3}+i\overset{3}{\underset{m=1}{\sum}}[\xb_{m},\overline{\xb}_{m}]\bigg\|
=\bigg|\overline{w}_2w_3(1-iac)-\overline{w}_3w_2(1+iac)-2iab|w_2|^2\bigg|.
\end{eqnarray*}

Hence, the last equality with \eqref{x123} reduces
\eqref{ksf2} to
\begin{eqnarray*}
\bigg|\overline{w}_2w_3(1-iac)-\overline{w}_3w_2(1+iac)-2iab|w_2|^2\bigg|\leq|w_2|^2+|w_3|^2-|a|^2|w_2|^2-|b\overline{w}_2+c\overline{w}_3|^2
\end{eqnarray*}
Letting $w_3=0$ in the last inequality, one gets
\begin{eqnarray*}
2|ab||w_2|^2\leq|w_2|^2(1-|a|^2-|b|^2)
\end{eqnarray*}
which is equivalent to
\begin{eqnarray}\label{e15}
|a|+|b|\leq1.
\end{eqnarray}

Consequently, if $|a|+|b|>1$, then \eqref{e15} is not satisfied, and this proves the desired assertion.
\end{proof}

Now lets us provide more concrete examples of the parameters. Take $c=0$, $a=b=1/\sqrt{3}$, then one can see
that \eqref{bb5}, \eqref{e14} are satisfied, but $|a|+|b|=2/\sqrt{3}>1$.

\section*{Acknowledgement} The first named author (F.M)
thanks the Scientific and Technological Research Council of Turkey
(TUBITAK) for support and Harran University for kind hospitality.
Moreover, he and the forth named author (A.A.)  also acknowledge
Research Endowment Grant B (EDW B 0905-303) of IIUM and the MOSTI
grant 01-01-08-SF0079. Finally, authors would like to thank to an anonymous referee whose useful
suggestions and comments improve the content of the paper.

\end{document}